\documentclass[12pt]{article}
\usepackage[margin=1.4in]{geometry}
\usepackage{amssymb,amsmath}              
\usepackage{graphicx}
\usepackage{amsthm}
\usepackage{amssymb}
\usepackage{color}
\usepackage{epstopdf}
\renewcommand{\epsilon}{\varepsilon}
\DeclareMathOperator{\dvg}{div} 
\DeclareMathOperator{\dist}{dist} 
\DeclareMathOperator{\graph}{graph}

  \DeclareMathOperator{\area}{Area}
  \DeclareMathOperator{\length}{Length}
   \DeclareMathOperator{\TC}{TotalCurvature}
   \DeclareMathOperator{\CN}{Cone}

\def\R{\mathbb{R}}
\def\C{\mathbb{C}}
\def\N{\mathbb{N}}

\def\d{\delta}
\def\a{\alpha}
\def\e{\epsilon}

\def\r{\rho}
\def\n{\nu}

\def\s{\sigma}

\def\l{\lambda}
\def\k{\kappa}

\def\H{\mathcal{H}}

\def\wt{\widetilde}

\def\ov{\overline}

\newtheorem{theorem}{Theorem}[section]
\newtheorem{lemma}[theorem]{Lemma}

\newtheorem{remark}[theorem]{Remark}
\newtheorem{corollary}[theorem]{Corollary}
\newtheorem{definition}[theorem]{Definition}

\begin{document}
\begin{title}
{Density estimates for compact surfaces with total boundary curvature less than $4\pi$.}
\end{title}
\begin{author}{Theodora Bourni \and Giuseppe Tinaglia~\footnote{Partially supported by The Leverhulme Trust and EPSRC grant no. EP/I01294X/1}}\end{author}
\date{}
\maketitle
\begin{abstract}
In this paper we obtain density estimates for compact surfaces immersed in $\mathbb{R}^n$ with total boundary curvature less than $4\pi$ and with sufficiently small $L^p$ norm of the mean curvature, $p>2$. Our results generalize the main results in~\cite{eww1}. We then apply our estimates to discuss the geometry and topology of such surfaces.\end{abstract}

\section{Introduction}

In  \cite{eww1}, Ekholm, White and Wienholtz proved some density estimates to show that a classical minimal surface immersed in $\mathbb{R}^n$ with total boundary curvature at most $4\pi$ is  embedded up to and including the boundary. Their results were later generalized by Choe and Gulliver to minimal surfaces in Riemannian manifolds with sectional curvatures bounded above by a non-positive constant~\cite{cgu3}.

In this paper we show that the density estimates in~\cite{eww1} hold for compact surfaces immersed in $\mathbb{R}^n$ with total boundary curvature less than $4\pi$ and with sufficiently small $L^p$ norm of the mean curvature, $p>2$. In particular, it follows that such surfaces are also embedded up to and including the boundary.

In Section~\ref{appl} we prove that when $n=3$, if the $C^{0,\alpha}$ seminorm of the mean curvature is bounded, then $|A|$ is bounded. We use this result to give a uniform upper bound for the genus of such surfaces. See~\cite{tin3} for the minimal case.

\section{Extended monotonicity formula}

We begin this section by introducing the notations and definitions that we will use in the rest of the article.

\begin{definition}\label{branchimm} Let $M$ be a smooth surface. Then by a branched immersion of $M$ into $\R^n$ we mean a $C^1$ map

\[u: M\to \R^n\]
such that for any $p\in{M}$ there exists a neighborhood $U$ of $p$ such that either
\begin{itemize}
\item[(i)] $u|_U$ is an immersion, i.e. the matrix $[ Du(p)]$ has rank 2, or
\item[(ii)] $u$ has a branch point at $p$, i.e. there exists an integer $m\ge2$ such that under a diffeomorphism of $M$ and a rotation of $\R^n$
\[u|_U=(z^m, f(z))\in\C\times\R^{n-2} \simeq \R^{n}\]
for some $f:\C\to\R^{n-2}$, where we identify $(x,y)\in U$ with $z=x+iy\in \C$.
\end{itemize}
\end{definition}

From now on we let $M$ be  a compact smooth surface with smooth boundary $\Gamma$ and we consider   a branched immersion of $M$ into $\R^n$
 \[u: M\to \R^n\]
 such that
 \[u|_\Gamma \text{  is  } 1-1.\]
 Hence for any $p\in\ov{M}$, either $(i)$ or $(ii)$ of Definition \ref{branchimm} holds.

Abusing the notation, we will call the image of the surface and that of its boundary, under the map $u$, $M$ and $\Gamma$ as well. We let $B_r(x)$ denote the $n$-dimensional ball of radius $r$ centered at $x$ and for a ball of a different dimension we use an exponent to indicate the dimension.

In what follows we review some facts about the density of a surface, see for example \cite{si1}.

\begin{definition}\label{density} The density function is
\[\Theta: \R^n\to\R\]
\[\Theta(M,x)=\lim_{r\to 0}\frac{\area(M\cap B_r(x))}{\pi r^2},\]
where the area is defined using the area formula (cf. \cite[\S8]{si1} ), i.e. letting $A=u^{-1}(B_r(x))\subset M$
\[\area(M\cap B_r(x))=\int_{\R^n}\H^{0}(u^{-1}(y)\cap A)d\H^2(y)=\int_A Ju d\H^2.\]
Here $Ju$ is the Jacobian of $u$
\[Ju(z)=\sqrt{\det(du_z)^*\circ(du_z)}\]
with $du_z:T_zM\to\R^2$ being the induced linear map.
\end{definition}

\begin{remark}\label{embrmk}
The above definition along with the fact that $u|_\Gamma$ is 1-1, implies
\begin{itemize}
\item $\Theta(M,x)\in \N\setminus\{0\}$ when $ x\in M\setminus \Gamma$ (Interior),
\item $\Theta(M,x)+\frac12\in \N\setminus\{0\}$ when $ x\in \Gamma$ (Boundary).\end{itemize}
In particular, we have that $M$ is embedded around a point $x\in M$ if and only if
\begin{itemize}
\item $\Theta(M,x)=1$ (namely $<2$) when $ x\in M\setminus \Gamma$ (Interior),
\item $\Theta(M,x)=\frac12$ (namely $<
\frac32$) when $ x\in \Gamma$ (Boundary).
\end{itemize}
\end{remark}


We now recall the definition of the mean curvature vector.

\begin{definition}\label{meancurv}Let $M$ be a branched immersion. Let $H$ be a normal vector field
\[ H: M\to\R^n\]
that is locally integrable. We say that $M$ has mean curvature vector $H$, if
\begin{equation*}
\int_M \dvg_MXdA=\int_{M}X\cdot H dA\end{equation*}
 for every $C^1$ vector field $X$ with compact support in $\R^n\setminus \Gamma$.

\end{definition}

\begin{remark}\label{dvgthmrmk} Note that even if the mean curvature vector exists, in general it is not true that
\begin{equation}\label{dvgthmeq}
\int_M \dvg_MXdA=\int_{M}X\cdot H dA+\int_\Gamma X\cdot \n_M ds
\end{equation}
for every $C^1$ vector field $X$ in $\R^n$, where $\n_M$ here is the outward pointing normal to the boundary (tangent to $M$) and the integral over $\Gamma$ is with respect to the one-dimensional Hausdorff measure.

However, if $M$ is a branched immersion having $H\in L^p(M)$ for some $p>2$ and $\Gamma$ is $C^{1,\a}$ for some $\a\in(0,1]$, then equation~\eqref{dvgthmeq} holds. This is because then the immersion $u$ is $C^{1,\beta}$ up to and including the boundary, where $\beta=\min\{\a, 1-n/p\}$ (cf. \cite[8.36]{gt1}).
\end{remark}

Let $M$ be a branched immersion such that equation~\eqref{dvgthmeq} holds, meaning that $M$ has mean curvature vector $H$ and equation~\eqref{dvgthmeq} holds. Let $\nu_M$ be the outward pointing unit normal to $\partial M$,  and for $x_0\in \R^n$ and $r>0$ let  $B(r)=B_r(x_0)$. Then arguing as in the case when $\partial M=\emptyset$ (cf. \cite[\S 17]{si1}) and keeping track of the boundary term in equation~\eqref{dvgthmeq} we have the following monotonicity formula with a boundary term.
\begin{equation}\label{monotonicitywbdry}\begin{split}
\frac{d}{dr}&\left(\frac{\area(M\cap B(r))}{r^2}\right)=\frac{d}{dr}\int_{M\cap B(r)}\frac{|D^\perp\r|^2}{\r^2}dA\\&\hskip2cm+ r^{-3}\int_{M\cap B(r)}(x-x_0)\cdot H dA-r^{-3}\int_{\partial M\cap B(r)}(x-x_0)\cdot \nu_M d\H^1\end{split}
\end{equation}
where $\r=\r(x)=|x-x_0|$.

  For $x_0\in\R^n$  let $E$ be the exterior cone over $\Gamma$ with vertex $x_0$, i.e.
 \[E=\{x_0+t(x-x_0): x\in\Gamma, t\ge 1\},\]
 $M'= M\cup E$ and define
\[m(r)=\frac{\area(M'\cap B(r))}{ r^2}.\]

By use of the standard monotonicity formula \eqref{monotonicitywbdry}, in \cite{eww1} it is proven that when $H=0$, $m(r)$ is a non decreasing function of $r$. Here we show that $m(r)$ still satisfies a monotonicity property even when the surface is not necessarily minimal.

Using equation \eqref{monotonicitywbdry} with $M$ replaced by $E$, gives
\[
\frac{d}{dr}\left(\frac{\area(E\cap B(r))}{r^2}\right)=-r^{-3}\int_{\partial E\cap B(r)}(x-x_0)\cdot \nu_E d\H^1,
\]
since $D^\perp\r\equiv 0$ on $E$ and
\[(x-x_0)\cdot H=0\quad\forall x\in E.\]
Hence:
\begin{equation}\label{extmon1}\begin{split}
m'(r)=&\frac{d}{dr}\left(\frac{\area(M\cap B(r))}{r^2}\right)+\frac{d}{dr}\left(\frac{\area(E\cap B(r))}{ r^2}\right)\\
=&\frac{d}{dr}\int_{M\cap B(r)}\frac{|D^\perp\r|^2}{\r^2}dA+ r^{-3}\int_{M\cap B(r)}(x-x_0)\cdot H dA\\
&-r^{-3}\int_{\Gamma\cap B(r)}(x-x_0)\cdot \nu_M d\H^1-r^{-3}\int_{\Gamma\cap B(r)}(x-x_0)\cdot \nu_E d\H^1.\end{split}
\end{equation}
Furthermore, for $x\in \Gamma$, we have that
\[(x-x_0)\cdot(-\n_E)=\max\{(x-x_0)\cdot\nu:\nu\in N_x\Gamma\}\]
where $N_x\Gamma$ is the normal space to $\Gamma$ at $x$, thus
\[(x-x_0)\cdot (\nu_E+\nu_M)\le0.\]
Hence equation \eqref{extmon1} gives the extended monotonicity formula
\begin{equation}\label{extmon2}\begin{split}
m'(r)&\ge \frac{d}{dr}\int_{M\cap B(r)}\frac{|D^\perp\r|^2}{\r^2}dA+ r^{-3}\int_{M\cap B(r)}(x-x_0)\cdot H dA\\
&\ge  r^{-3}\int_{M\cap B(r)}(x-x_0)\cdot H dA.\end{split}
\end{equation}

We now introduce another definition.

\begin{definition} A branched immersion $M$ satisfies the property \emph{(P)} if $M$ has mean curvature vector $H$ and there exist $\a\in(0,1]$ and a constant $\Lambda\ge0$ such that
\begin{equation*}\label{propp}
\int_{M\cap B(r)}|H|dA\le \a \Lambda r^{\a+1} m(r),\,\forall r>0.\tag{P}
\end{equation*}
\end{definition}
%

\begin{lemma}\label{extmonLinftylemmaP}
Let $M$ be a branched immersion such that equation~\eqref{dvgthmeq} holds and satisfying the property \emph{(P)},  then $e^{\Lambda r^\a} m(r)$ is an increasing function of $r$, that is
\begin{equation}\label{extmonLinfty}
e^{\Lambda r^\a} m(r)-e^{\Lambda \s^\a} m(\s)\ge 0,  \text{ for } 0<\s<r<\infty.
\end{equation}
\end{lemma}
\begin{proof}
The inequality in \eqref{extmon2} gives
\begin{equation*}
m'(r)\ge -r^{-2}\int_{M\cap B(r)}|H| dA\ge -\a\Lambda r^{\a-1}  m(r)\implies
\end{equation*}
\[\frac{d}{dr}\left(e^{\Lambda r^{\a}} m(r)\right)\ge0\]
Integrating this from $\s$ to $r$ proves the lemma.
\end{proof}

 We would like to compare $\lim_{r\to0}m(r)$, which gives  the density of $M$ at $x_0$ with $\lim_{r\to\infty}m(r)$, which we will later see  can be estimated using the total curvature of the boundary $\Gamma$. However, the previous lemma is not helpful when $r$ is going to infinity. Thus, we have to find a more efficient way to compare the ratios for large $r$.

 Let $r_0=r_0(M)$ denote the extrinsic diameter of $M$, that is
\[r_0(M)=\max_{x,y\in M}|x-y|.\]
Recall that we always assume that $M$ is compact and thus $r_0<\infty$.

\begin{lemma}\label{extmonLinftylemmabigP}
Let $M$ be a branched immersion such that equation~\eqref{dvgthmeq} holds and satisfying the property \emph{(P)}, then
 for any $r\ge r_0$:
\begin{equation}\label{extmonLinftybig}
m(r)\ge m(r_0)\left(1-\frac{\a\Lambda r_0^\a}{2}\left(1-\frac{r_0^2}{r^2}\right)\right).
\end{equation}
\end{lemma}
\begin{proof}
Integrating the inequality in \eqref{extmon2} from $r_0$ to $r$ we have (using integration by parts)
\begin{equation*}\begin{split}
m(r)-m(r_0)\ge &\frac12\int_{M\cap B(r)}(x-x_0)\cdot H\left(\frac{1}{r_0^2}-\frac{1}{r^2}\right)dA\\
=&\frac12\int_{M\cap B(r_0)}(x-x_0)\cdot H\left(\frac{1}{r_0^2}-\frac{1}{r^2}\right)dA\\
\ge&-\frac12\a\Lambda r_0^{2+\a}m(r_0)\left(\frac{1}{r_0^2}-\frac{1}{r^2}\right)
\end{split}\end{equation*}
\end{proof}

\section{Density Esimates and Embeddedness}\label{denestsec}

We now state and prove our main density estimate.
\begin{theorem}\label{dencur} Let $M$ be a branched immersion in $\R^n$ such that equation~\eqref{dvgthmeq} holds and satisfying the property~\emph{(P)}. Then, for $x_0\in\R^n$

\begin{equation}\label{dencurLinfty}
\Theta(\CN_{x_0}\Gamma,x_0)\ge e^{-\Lambda r_0^\a}\left(1-\frac{\a\Lambda r_0^\a}{2}\right)\Theta(M, x_0)
\end{equation}
\end{theorem}
\begin{proof}
 Let $E$ be the exterior cone over $\Gamma$ with vertex $x_0$, $M'=M\cup E$ and $m(r)=\area(M'\cap B(r))/ r^2$, as before. Then by the extended monotonicity formula (Lemmas \ref{extmonLinftylemmaP}, \ref{extmonLinftylemmabigP}) we get for $r< r_0<R$
 \begin{equation}\label{mr1}
m(R)\ge m(r_0)\left(1-\frac{\a\Lambda r_0^\a}{2}\left(1-\frac {r_0^2}{R^2}\right)\right)\ge e^{\Lambda (r^\a-r_0^\a)}\left(1-\frac{\a\Lambda r_0^\a}{2}\left(1-\frac {r_0^2}{R^2}\right)\right) m(r)
\end{equation}
Letting $r\to 0$ and $R\to\infty$ we get
\[\begin{split}\pi\Theta(\CN_{x_0}\Gamma,x_0)&=\lim_{R\to\infty}m(R)\\
&\ge e^{-\Lambda r_0^\a}\left(1-\frac{\a\Lambda r_0^\a}{2}\right)\lim_{r\to0} m(r)=e^{-\Lambda r_0^\a}\left(1-\frac{\a\Lambda r_0^\a}{2}\right)\pi\Theta(M, x_0).\end{split}\]

We note that, as in \cite{eww1}, in the case when $x_0\in\Gamma$,
\[\lim_{r\to0} m(r)= \pi\Theta(M, x_0)\]
because
\[\lim_{r\to0} m(r)=\pi\Theta(M,x_0)+\lim_{r\to0} \frac{\area(E\cap B(r))}{ r^2}\]
and the last term is $0$, since $\Gamma$ is $C^1$.
\end{proof}

\begin{remark}\label{massupbd} If in the inequality \eqref{mr1}  in the proof of Theorem~\ref{dencur}, we only let $R\to\infty$ we get the following more general conclusion about the area ratios:
\begin{equation}\label{dencurLinftym}
\pi\Theta(\CN_{x_0}\Gamma,x_0)\ge e^{-\Lambda (r_0^\a-r^\a)}\left(1-\frac{\a\Lambda r_0^\a}{2}\right)m(r),\,\forall r\le r_0.
\end{equation}
\end{remark}

The following theorem is a result from~\cite{eww1} that compares the density of the cone with the total curvature of $\Gamma$ for points that are not in $\Gamma$.

\begin{theorem}\label{dencur0}
Let $\Gamma$ be a closed curve in $\R^n$ and $x_0\notin\Gamma$. Then
\[\length \Pi_{x_0}\Gamma\le\TC\Gamma\]
where
\[\Pi_{x_0}(x)=x_0+\frac{x-x_0}{|x-x_0|}.\]
\end{theorem}
The theorem equivalently states the following: For the cone $\CN_{x_0}\Gamma=\{x_0+t(x-x_0):x\in\Gamma, 0\le t\le1\}$, since the ratios
\[
\frac{\area(\CN_{x_0}\Gamma\cap B(r))}{\pi r^2}=\frac{\length(\CN_{x_0}\Gamma\cap\partial B(r))}{2\pi r}=\frac{\length\Pi_{x_0}\Gamma}{2\pi}
\]
are constant, we have
\begin{equation}\label{conedensity}
\begin{split}\Theta(\CN_{x_0}\Gamma, x_0)&=\lim_{r\to0}\frac{\area(\CN_{x_0}\Gamma\cap B(r))}{\pi r^2}\\
&=\frac{\length \Pi_{x_0}\Gamma}{2\pi}\le\frac{\TC\Gamma}{2\pi}.\end{split}
\end{equation}

Using Theorems \ref{dencur0} and Theorem \ref{dencur}  gives embeddedness at interior points.

\begin{corollary}\label{intemb}
Let $\Gamma$ be a $C^1$ closed curve in $\R^n$ with total curvature at most $(4-\e)\pi$, $\e\in(0,2]$. Let $M$ be a branched immersion in $\R^n$ with boundary $\Gamma$ such that equation~\eqref{dvgthmeq} holds. If $H\in L^p(M)$ for some $p\in(2,\infty]$ then there exists $\d=\d(\e)\in (0,1)$ such that if
\[C(p)\|H\|_{L^p(M)} r_0^\a<\d\]
then the interior of $M$ is embedded.
Here $C(p)=1$ and $\alpha=1$ if $p=\infty$ and $C(p)=\frac{2p}{p-2}\left(\frac2\pi\right)^{1/p}$ and $\alpha=1-2/p$ if $p\in(2,\infty)$. \end{corollary}

\begin{proof} 
Let $x_0\in M$. Since $\delta$ is less than one, then $M$ satisfies the property (P)  with $\Lambda=C(p)\|H\|_{L^p(M)}$ and $\alpha=1-2/p$, see Lemma~\ref{inftyP} and Corollary~\ref{pP} in the Appendix. Combining Theorems \ref{dencur0} and Theorem \ref{dencur} we get that
\begin{equation*}
\Theta(M, x_0)\le \frac{e^{\Lambda r_0^\a}}{1-\frac{\a\Lambda r_0^\a}{2}}\Theta(\CN_{x_0}\Gamma, x_0)\le
\frac{e^{\d}}{1-\frac{\a\d}{2}}\frac{\TC\Gamma}{2\pi}.
\end{equation*}

The  corollary then holds,  provided that:
\begin{equation}\label{dinfty}
\frac{e^\d}{2-\a\d}(4-\e)<2
\end{equation}
since this estimate implies that
\[\Theta(M, x_0)<2\]
(cf. Remark \ref{embrmk}).
\end{proof}

In order to obtain embeddedness at boundary points we need to recall another result in~\cite{eww1}.

\begin{theorem}\label{dencur0b} Let $\Gamma$ be a simple closed curve in $\R^n$ with finite total curvature and $x_0\in\Gamma$. Then
\[\length \Pi_{x_0}\Gamma\le\TC\Gamma-\pi-\theta\]
where
\[\Pi_{x_0}(x)=x_0+\frac{x-x_0}{|x-x_0|}\]
and $\theta$ is the exterior angle to $\Gamma$ at $x_0$.
\end{theorem}

\begin{corollary}\label{smbdryemb}
Let $\Gamma$ be a $C^1$ closed curve in $\R^n$ with total curvature at most $(4-\e)\pi$, $\e\in(0,2]$. Let $M$ be a branched immersion in $\R^n$ with boundary $\Gamma$
and such that equation~\eqref{dvgthmeq} holds. If $H\in L^p(M)$ for some $p\in(2,\infty]$ then there exists $\d=\d(\e)\in (0,1)$ such that if
\[C(p)\|H\|_{L^p(M)} r_0^\a<\d\]
then $M$ is embedded up to and including the boundary.
Here $C(p)=1$ and $\alpha=1$ if $p=\infty$ and $C(p)=\frac{2p}{p-2}\left(\frac2\pi\right)^{1/p}$ and $\alpha=1-2/p$ if $p\in(2,\infty)$.\end{corollary}
\begin{proof}
By Corollary~\ref{intemb}, it suffices to prove that if $\d$ is sufficiently small, then $M$ is embedded at boundary points. Let $x_0\in\Gamma$, then using Theorems \ref{dencur0b} (with $\theta=\pi$, since $\Gamma$ is $C^1$) and Theorem \ref{dencur} we get
\begin{equation*}
\Theta(M, x_0)\le \frac{e^{\Lambda r_0^\a}}{1-\frac{\a\Lambda r_0^\a}{2}}\Theta(\CN_{x_0}\Gamma, x_0)\le
\frac{e^{\Lambda r_0^\a}}{1-\frac{\a\Lambda r_0^\a}{2}}\left(\frac{\TC\Gamma}{2\pi}-\frac12\right).
\end{equation*}

Therefore  the corollary holds, provided that $\e$ is small enough to satisfy the following
\begin{equation}\label{dinftybd}
\frac{e^\d}{2-\a\d}(3-\e)<\frac32
\end{equation}
 since this estimate implies that
\[\Theta(M, x_0)<\frac32\]
 (cf. Remark \ref{embrmk}).
\end{proof}

As in \cite{eww1}, our results extend to the case when the boundary $\Gamma$ is piecewise $C^1$.
In particular the extended monotonicity formula \eqref{extmon2} and thus Theorems \ref{dencur}, \ref{intemb} and \ref{smbdryemb} hold with the $C^1$ condition on the boundary $\Gamma$ replaced by the weaker condition of being piecewise $C^1$.

\begin{theorem}\label{corbdryemb}
Let $\Gamma$ be a $C^1$ closed curve in $\R^n$ with total curvature at most $(4-\e)\pi$, $\e\in(0,2]$. Let $M$ be a branched immersion in $\R^n$ with boundary $\Gamma$ and such that equation~\eqref{dvgthmeq} holds.  If $H\in L^p(M)$ for some $p\in(2,\infty]$ then there exists $\d=\d(\e)\in (0,1)$ such that if
\[C(p)\|H\|_{L^p(M)} r_0^\a<\d\]
then the following holds.
\begin{itemize}
\item[1.] If $x_0$ is a point on $\Gamma$ with exterior angle $\theta$, then the density $\Theta(M,x_0)$ is either $\frac12+\frac{\theta}{2\pi}$ or $\frac12-\frac{\theta}{2\pi}$
\item[2.] If $x_0$ is a cusp ($\theta=\pi$) then $\Theta(M,x_0)=0$, unless $\Gamma$ lies in a plane
\item[3.]$M$ is embedded up to and including the boundary.
\end{itemize}
Here $C(p)=1$ and $\alpha=1$ if $p=\infty$ and $C(p)=\frac{2p}{p-2}\left(\frac2\pi\right)^{1/p}$ and $\alpha=1-2/p$ if $p\in(2,\infty)$.
\end{theorem}
We skip the proof since it is identical with that of the minimal case (cf. \cite[Theorem 4.1]{eww1}) with the difference that instead of Theorem 1.3 in \cite{eww1} we have to use Theorem \ref{dencur} here.

 We note that in all of the previous results, the proofs provide a specific value for $\d=\d(\e)$.

\section{Applications}\label{appl}

When $n=3$, using the density estimates we can prove a curvature estimate which is a generalization of Theorem 0.1 in~\cite{tin3}.

We begin by introducing a few definitions.
\begin{definition}\label{c1a convergence}
 Let $\Gamma\subset \R^{n}$ be a simple closed $C^{2,\a}$ curve, 
 i.e  for each $x\in \Gamma$ there exists an $r=r(x)>0$ such that
$$
M\cap B_{r}(x) = \graph u_{x}\cap  B_{r}(x),
$$
where $u_{x}\in C^{2,\alpha}((x+L_{x})\cap \overline B_{r}(x);L_{x}^{\perp})$ for some
$1$-dimensional subspace $L_{x}$ of $\R^{n}$ and where $\graph
u_{x}=\{\xi+u_{x}(\xi):\xi\in (x+L_{x})\cap  \ov B_{r}(x)\}$.\\
Let
\[\kappa(\Gamma,r,x) = \inf\|u_x\|^*_{2,\a,(x+L_x)\cap B_r(x)}\]
where the infimum is taken over all choices of subspaces $L_{x}$ and corresponding representing functions $u_{x}$ and where for a $C^{2,\a}$ function $u: (-r,r)\to \R^{n-1}$, $\|u\|^*_{2,\a, (-r,r)}$ denotes the scale invariant $C^{2,\a}$ norm of $u$, i.e.
\[\begin{split}\|u\|^*_{2,\a, (-r,r)}=\frac{1}{r}\sup_{y\in (-r,r)}|u(y)|&+\sup_{y\in (-r,r)}|Du(y)|
+r\sup_{y\in (-r,r)}|D^2u(y)|\\
&+r^{1+\a}\sup_{y,z\in (-r,r), y\ne z}\frac{|D^2u(y)-D^2u(z)|}{|y-z|^\a}.\end{split}\]
The  quantity $\kappa(\Gamma,r,x)$ is a way to quantify the regularity of $\Gamma$ around $x$.
\end{definition}
\begin{definition}
 Let $\Gamma\subset \R^{n}$ be a $C^{2,\a}$ simple closed curve. 
 For each $x\in\Gamma$ let
 \[{\r}(x)=\sup\{r:\k(\Gamma,r,x)\le1\}.\]
Finally we define $\r=\r(\Gamma)$ by
 \[\r(\Gamma)=\min\{{\r}(x):x\in\Gamma\}.\]
 Note that for each $x\in \Gamma, \r(x)>0$ and hence $\r(\Gamma)>0$.
\end{definition}
In particular, for any $x\in\Gamma$, $B_{\r(\Gamma)}(x)\cap\Gamma$ can be written as a graph with gradient less than one. Note that if $\r(\Gamma)=\infty$ then $\Gamma$ is a straight line. This is because \[r{|D^2u(x)|}\le 1,\, \forall r \implies |D^2u(x)|=0,\]

\begin{definition}\label{Pea} Given $\e\in(0,2]$, $\a\in(0,1]$ we say that $(\Gamma, M)$ are in the $P(\e,\a)$ class, if the following hold:
\begin{itemize}
\item [(i)] $\Gamma$ is a simple, closed $C^{1,\a}$ curve in $\R^n$ with total curvature at most $(4-\e)\pi$.
\item[(ii)] $\partial M=\Gamma$.
\item [(iii)] $M$ is a branched immersion having mean curvature $H\in L^p(M)$, where $p=2/(1-\a)$.
\item [(iv)] $\Lambda r_0^\a<\d$, where $r_0= r_0(M)$ the diameter of $M$ and
\begin{itemize}
\item If $\a=1\Leftrightarrow p=\infty$, $\Lambda=\|H\|_{L^\infty(M)}$
\item If $\a\in(0,1)\Leftrightarrow p\in(2,\infty)$, $\Lambda=\frac{2p}{p-2}\left(\frac2\pi\right)^{1/p}\|H\|_{L^p(M)}$
\end{itemize}
and $\d=\d(\e)\in(0,1)$ is such that
\[\frac{e^\d}{1-\d}<\frac{4}{4-\e}.\]
\end{itemize}
\end{definition}

Note that if $(\Gamma, M)\in P(\e,\a)$, then equation~\eqref{dvgthmeq} holds (see Remark \ref{dvgthmrmk}), $M$ satisfies the property  (P)  and Theorems \ref{dencur}, \ref{intemb},  \ref{smbdryemb} hold, so that $M$ is embedded and $C^{1,\a}$ up to and including the boundary. Finally the area ratio bounds given in Remark \ref{massupbd} hold.

\begin{theorem}\label{mainthm}
 Given $\lambda>0$, $\gamma>0$ and $\a\in(0,1]$ there exists a constant  $C=C(\lambda,\gamma,\a)$  such that the following holds.

Let $\Gamma$ be a $C^{2,\alpha}$ closed curve in $\R^3$ and let $M$ be an orientable branched immersion in $\R^3$ with boundary $\Gamma$.
If $(\Gamma,M)\in P(\e,\alpha)$, for some $\e\in(0,2]$, and, in addition, $\r(\Gamma)/r_0\ge\gamma$ and the mean curvature $H$ of $M$ is in $C^{0,\a}(M)$ satisfying $r_0^{1+\a}[H]_{C^{0,\alpha}(M)}\le\lambda$, then
$$\sup_{M}r_0|A|\leq C,$$
where $|A|$ is the norm of the second fundamental form.
\end{theorem}

In the above theorem $[H]_{C^{0,\alpha}(M)}$, denotes the $\a$-seminorm of $H$, i.e.
\[[H]_{C^{0,\alpha}(M)}=\sup_{x,y\in M, x\ne y}\frac{|H(x)-H(y)|}{|x-y|^\a}.\]

\begin{proof}
The hypotheses on $\Gamma$ and $H$ along with standard PDE estimates ($W^{2,p}$ estimates, cf. \cite[Theorem 9.15]{gt1} and the Schauder estimates) imply that $M$ is $C^2$ up to and including the boundary (in particular M in $C^{2,\beta}$ for any $\beta<\a$). This implies that for a given surface, $|A|$ is bounded.
With this theorem we are showing that the bound does not depend on the surface but only on $\lambda,\a $ and $\gamma$.

Arguing by contradiction, suppose that the statement is false. Then we can find a sequence of $C^{2,\alpha}$ connected simple closed curves $\Gamma_n\subset \mathbb{R}^3$   with total curvature at most $(4-\e_n)\pi$, $\e_n\in(0,2]$, and surfaces $M_n$ with diameter $r_0(M_n)= r_{0,n}$ and mean curvature $H_n\in C^{0,\a}(M_n)$ such that the following hold.
 The pairs $(\Gamma_n, M_n)\in P(\e,\a)$, in particular $\partial M_n=\Gamma_n$ and $\Lambda_n r_{0,n}^\a<\d(\e_n)<1$. Here $\Lambda_n$, $\d(\e_n)$ are the corresponding $\Lambda$ and $\d$ as in Definition \ref{Pea}. Moreover $\r(\Gamma_n)/r_{0,n}\ge\gamma$, $r_{0,n}^{1+\a}[H_n]_{C^{0,\alpha}(M_n)}\le\lambda$ and there exists
$p_n \in M_n$ such that
$$r_{0,n}|A_n(p_n)|=r_{0,n}\max_{\ov M_n}|A_n|> n,$$

We consider the sequence of surfaces $M_n'= \eta_{p_n,|A(p_n)|}(M_n)$ (where $\eta_{x,\l}(y)=\l(y-x) $ that is a rescaling and a translation). Note that for these new surfaces we have the following:
Their second fundamental forms are uniformly bounded, namely
\[
\sup_{M_n'} |A'_n|\le|A_n(p_n)|^{-1}\sup_{M_n} |A_n|= 1
\]
 and their mean curvatures satisfy
 \[[H'_n]_{0,\a}=|A_n(p_n)|^{-(1+\a)}[H_n]_{0,\a}\le\l(r_{0,n}|A_n(p_n)|)^{-(1+\a)}\to 0\]
 and
 \[\|H'_n\|_{L^p(M_n')}=\frac{ r_{0,n}^{1-2/p}}{(|A_n(p_n)| r_{0,n})^{1-2/p}}\|H_n\|_{L^p(M_n)}<\d(\e_n)(r_{0,n}|A_n(p_n)|)^{-\a}\to 0.\]

 A standard compactness argument, using Schauder estimates, shows that a subsequence of  $M'_n$ converges $C^2$  to an orientable properly embedded minimal surface $\widetilde{M}$ such that
 \begin{itemize}
 \item[(i)]  if ${\dist}_{M_n}(\Gamma, p_n)|A_n|\to\infty$,
 then $\wt{M}$ is complete, or
 \item[(ii)]  if ${\dist}_{M_n}(\Gamma, p_n)|A_n|$ is bounded, then $\wt{M}$ has a boundary $B$.
\end{itemize}
Moreover if case (ii) holds then $B$ is a straight line, since for $\Gamma'_n=\partial M'_n$ we have that
\[\r(\Gamma'_n)=|A_n(p_n)|\r(\Gamma_n)\ge\gamma |A_n(p_n)| r_{0,n}\to\infty.\]

 Using the area estimates derived in Theorem \ref{dencur} (see also Remark \ref{massupbd}) we can take the cone at infinity of $\wt{M}$ and arguing exactly as in~\cite[Theorem 0.1]{tin3} we can conclude that $\widetilde{M}$ is either a plane or a  half-plane. However, the norm of the second fundamental form of $\widetilde{M}$  has value one at the origin, in particular $\widetilde{M}$ is neither a plane nor a  half-plane. This gives a contradiction and proves the estimate.
  \end{proof}

In the following lemma the dimension n, of the ambient space $\R^n$, need not be equal to 3.

\begin{lemma}\label{genusthm}
Given $\Delta>0$ there exists a constant  $N=N( \Delta)$  such that the following holds. If $(\Gamma,M)\in P(\e,\a)$ and $M$ is $C^2$ up to and including the boundary with $r_0\sup_M|A|<\Delta$, then
$M$ has genus less than $N$.\end{lemma}

\begin{proof}
The Gauss-Bonnet Theorem states that $$\int_{\Gamma} \vec{k}\cdot\vec{\nu}_Mds+\int_{M}K_M=2\pi\chi(M),$$ where $\vec{k}$ is the curvature vector of the curve $\Gamma$, $\vec{\nu}_M$ is the exterior normal of $M$, $K_M$ the Gauss curvature and $\chi(M)$ is the Euler characteristic of $M$. The first integral is bounded in absolute value by the total curvature of $\Gamma$. By the Gauss equation, the second integral can be bounded as follows.
\[
\left|\int_{M}K_M \right|\leq \frac12 \left(\int_{M} |A|^2 +\int_{M}H^2\right).
\]
Note that $|H|^2\le 2|A|^2$. Using the assumption on $|A|$ and Theorem \ref{dencur}, in particular Remark \ref{massupbd} we get
\[\begin{split}\frac12\int_M|H|^2+ |A|^2&\le\frac32\int_{M\cap B(r_0)}|A|^2\le \Delta^2 r_0^{-2}m(r_0) r_0^2\\
&\le \frac32\frac{\Delta^2 \pi}{1-\frac{\a\Lambda r_0^\a}{2}} \Theta(\CN_{x_0}\Gamma, x_0)
\le\frac{3\Delta^2 \pi}{2-\a\d}\frac{4-\e}{2}\leq 3\pi \Delta^2.
\end{split}\]
  Hence $|\chi(M)|$ is bounded and so is the genus. Note that $\a\d<1$ and $\epsilon\leq 2$.
  \end{proof}

\begin{corollary}
 Given $\lambda>0$, $\gamma>0$ and $\a\in(0,1]$ there exists a constant  $N=N(\lambda,\gamma,\a)$  such that the following holds.

Let $\Gamma$ be a $C^{2,\alpha}$ closed curve in $\R^3$ and let $M$ be an orientable branched immersion in $\R^3$ with boundary $\Gamma$.
If $(\Gamma,M)\in P(\e,\alpha)$, for some $\e\in(0,2]$, and, in addition, $\r(\Gamma)/r_0\ge\gamma$ and the mean curvature $H$ of $M$ is in $C^{0,\a}(M)$ satisfying $r_0^{1+\a}[H]_{C^{0,\alpha}(M)}\le\lambda$, then $M$ has genus less than $N$.
\end{corollary}

\section{Appendix}
In this appendix we show that if $H\in L^p(M)$, $p>2$, and $\|H\|_{L^p(M)}$ is sufficiently small, then $M$ satisfies the property~(P), see~\cite{si1}.

\begin{lemma}\label{inftyP}
If $H\in L^\infty(M)$, then setting $\|H\|_\infty=\|H\|_{L^\infty(M)}$ we have for any $r>0$:
\[\int_{M\cap B(r)}|H| dA\le \|H\|_\infty \area(M\cap B(r))\le \|H\|_\infty m(r) r^2\]
so $H$ satisfies \emph{(P)} with $\a=1, \Lambda=\|H\|_\infty$.
\end{lemma}

Assume now that $H\in L^p(M)$ for some $p\in(2,\infty)$.

\begin{lemma}\label{extmonLp1lemma}
If $H\in L^p(M)$, with $p>2$, then for any $0<\s<r<\infty$
\begin{equation}\label{extmonLp1}
m(\s)^{1/p}\le m(r)^{1/p}+ \frac{\|H\|_p}{p-2}\left(r^{1-2/p}-\s^{1-2/p}\right)
\end{equation}
where $\|H\|_p=\|H\|_{L^p(M)}$.

If we also assume that $x_0\in \ov M$ and
\[\|H\|_p r^{1-2/p}\le \frac{p-2}{2}\left(\frac\pi2\right)^{1/p}\]
then

\begin{equation*}\label{lowbd}
m(r)^{1/p}\ge\frac12 \left(\frac\pi2\right)^{1/p}.
\end{equation*}

\end{lemma}

\begin{proof}
By  using Holder inequality in the second inequality in \eqref{extmon2}, gives
\begin{equation*}
m'(r)\ge  -r^{-2}\int_{M\cap B(r)}|H| dA\ge -\|H\|_p r^{-2/p}m(r)^{1-1/p}\implies
\end{equation*}
\[ p\frac{d}{dr}\left(m(r)^{1/p}\right)\ge- \|H\|_p  r^{-2/p}\implies\frac{d}{dr}\left(m(r)^{1/p}\right)\ge-\frac{\|H\|_p }{p-2}\frac{d}{dr}r^{1-2/p}.\]
Integrating from $\s$ to $r$ implies \eqref{extmonLp1}.

If $x_0\in\ov M$, we have that
\[\lim_{r\to0} m(r)=\pi\Theta(x, M')\ge\pi /2.\]
If $x_0\in \ov M\setminus\Gamma$ then actually $\Theta(x_0,M')\ge 1$.
Hence, letting $\s\to 0$ in \eqref{extmonLp1} and using the assumption on $\|H\|_p$ we get that
\begin{equation*}
m(r)^{1/p}\ge\left(\frac\pi2\right)^{1/p}- \frac{\|H\|_p r^{1-2/p}}{p-2}\ge \left(\frac\pi2\right)^{1/p}-\frac12 \left(\frac\pi2\right)^{1/p}\ge\frac12 \left(\frac\pi2\right)^{1/p}.
\end{equation*}.
\end{proof}

\begin{corollary}\label{pP}
If $H\in L^p(M)$, with $p>2$ and
\[\|H\|_p r_0^{1-2/p}\le \frac{p-2}{2}\left(\frac\pi2\right)^{1/p}\]
then for any $r>0$
\[\begin{split}\int_{M\cap B_r}|H| dA\le& \|H\|_p \area(M\cap B(r))^{1-1/p}\le \|H\|_p r^{2-2/p}m(r) m(r)^{-1/p}\\
\le&2\left(\frac2\pi\right)^{1/p}\|H\|_p r^{2-2/p} m(r)\end{split}\]
and thus $H$ satisfies \emph{(P)} with $\a=1-2/p$ and
\[\Lambda=\frac{2p}{p-2}\left(\frac2\pi\right)^{1/p}\|H\|_p.\]
\end{corollary}


\bibliography{bill}
  \bibliographystyle{plain}
\end{document}